\documentclass[11pt]{article}

\usepackage{amsmath,amsthm,verbatim,amssymb,amsfonts,amscd, graphicx}
\usepackage[pdftex, dvipsnames]{xcolor}
\usepackage{mathabx}
\usepackage{graphics}
\usepackage{tikz-cd}
\theoremstyle{plain}
\newtheorem{theorem}{Theorem}

\newtheorem{lemma}{Lemma}

\theoremstyle{definition}

\DeclareMathOperator{\im}{im}
\DeclareMathOperator{\SL}{SL}

\begin{document}

\title{Building Groups From Restricted Diagrams of Groups}
\author{Nic Brody and Michael R. Klug}
\maketitle

\begin{abstract}
	We consider the problem of realizing a group as the fundamental group of a graph of groups where the vertex groups are restricted to certain classes (for example, coming from a certain finite list of groups, or having bounded geometric rank).  We show how this places restrictions on the possible groups that can be realized and we give a topological application of our results to the problem of constructing manifolds from a finite set of ``building blocks".   
\end{abstract}

\section{Introduction}

We prove two results concerning which groups can be obtained as fundamental groups of graphs of groups, where the vertex groups are restricted to coming from certain specific classes.  In section \ref{sec:realizing}, we prove that given a finite set of possible vertex groups $\mathcal{C}$, there are always infinitely many groups that cannot be realized as the fundamental groups of graphs of groups with vertex groups belonging to $\mathcal{C}$.  We show that this holds also for diagrams of groups (graphs of groups with non-injective edge homomorphisms).  This provides a different answer to a topological question of Martelli that was previously answered using group-theoretic methods by Freedman, which we discuss at the end of section \ref{sec:realizing}.  We also make some comments on variations of this question and the difficulty in dimension 4.  In section \ref{sec:rank}, we prove that the maximal rank of a free abelian subgroup of the fundamental group of a graph of groups is at most one larger than the maximal rank of the free abelian subgroups of each of the vertex groups.  We mention how this result does not hold for diagrams of groups. 

\section*{Acknowledgments}

 We would like to thank Benjamin Steinberg for explaining Lemma \ref{lem:n.i.} to us.  Michael would also like to thank the Max Planck Institute for Mathematics in Bonn, where some of this work was carried out.

\section{Realizing groups as fundamental groups of graphs of groups}\label{sec:realizing}

An abstract graph $\Gamma$ is a pair of sets $V$ and $E$ together with an involution $\overline{\phantom{x}} : E \to E$ with no fixed points and a map $\partial_0 : E \to V$.  The elements of $V$ are called vertices and the elements of $E$ are called edges.  An abstract graph determines an ordinary graph which we also denote by $\Gamma$, and we say that an abstract graph $\Gamma$ is connected if its corresponding (ordinary) graph is connected.  A graph of groups $\mathcal{G}$ is an abstract graph $\Gamma$ together with an assignment of a group $G_v$ for every vertex $v \in V$, and an assignment of a group $G_e$ and an injective homomorphism $f_e : G_e \to G_{\partial_0 e}$ for every edge $e \in E$, with the condition that $G_e = G_{\overline{e}}$.  The groups $G_v$ are called the vertex groups, the groups $G_e$ are called the edge groups, and the maps $f_e$ are called the edge maps.  A \emph{diagram of groups} $\mathcal{G}$ is exactly the same as a graph of groups except that the edge maps are not assumed to be injective.  Graphs of groups arise when studying the action of a group on a tree (or trying to construct such an action) while more general diagrams of groups arise when glueing together several topological spaces and computing the fundamental group of the result in terms of the fundamental groups of the pieces. 

We will assume going forward that all of our graphs of groups and diagrams of groups are connected (i.e. that they have connected underlying abstract graphs).  We will also assume that the underlying graphs of all of our graphs of groups and diagrams of groups are finite.

A graph of spaces $\mathcal{X}$ is an abstract graph $\Gamma$, together with an assignment of a topological space $X_v$ to each vertex $v \in V$, and an assignment of a topological space $X_e$ and a continuous map $f_e: X_e \to X_{\partial_0 e}$ for each edge $e \in E$.  The spaces $X_v$ are called the vertex spaces, the spaces $X_e$ are called the edge spaces, and the maps $f_e$ are called the edge maps.  A \emph{based} graph of spaces is a graph of spaces where all of the spaces have a chosen base point and all of the edge maps preserve these base points.  Given a graph of spaces $\mathcal{X}$, the total space of $\mathcal{X}$ is the quotient of the disjoint union of $\{ X_v : v \in V\}$ and $\{X-e \times I : e \in E \}$ by the identifications
\begin{align*}
	X_e \times I &\to X_{\overline{e}} \times I \\
	(x,t)        &\mapsto (x, 1-t)
\end{align*}
and 
\begin{align*}
	X_e \times 0 &\to X_{\partial_0 e} \\
	(x,0)        &\mapsto f_e(x)
\end{align*}

A based graph of spaces is defined similarly, but where all of the vertex spaces $X_v$ and edge spaces $X_e$ are based and where all of the edge maps $f_e : X_e \to X_{\partial_0 e}$ are based.  The total space of a based graph of spaces is defined similarly. The basepoints yield a naturally embedded copy of the graph $\Gamma$, on which we may select a basepoint.  Given a graph of groups $\mathcal{G}$, a graph of spaces for $\mathcal{G}$ is a based graph of spaces $\mathcal{X}$ with the same underlying graph as $\mathcal{G}$, with identifications $G_v \cong \pi_1(X_v)$ and $G_e \cong \pi_1(X_e)$ (using the base-points of $X_v$ and $X_e$), such that using these identifications the edge maps of $\mathcal{X}$ induce the edge maps of $\mathcal{G}$. 

For a graph of groups $\mathcal{G}$, the fundamental group of $\mathcal{G}$, denoted $\pi_1(\mathcal{G})$, is the fundamental group of a graph of spaces $\mathcal{X}$ for $\mathcal{G}$ (see \cite{scott_wall} where it is shown that this does not depend (up to isomorphism) on the particular choice of $\mathcal{X})$).  Note that a purely algebraic definition of $\pi_1(\mathcal{G})$ (or rather, two such definitions) is given by Serre in $\cite{trees}$ and we will use some properties of this definition later.  That these two definitions agree is the content of \cite{althoen} (see also the remarks at the end of \cite{higgins}).  As is remarked in \cite{althoen}, all of these properties of $\pi_1(\mathcal{G})$ extend to diagrams of groups and are again all equivalent.  

\begin{lemma} \label{lem:n.i.}
	Let $\mathcal{C}$ be a class of groups that is closed under homomorphic images (i.e. if $\phi : H_1 \to H_2$ is a homomorphism and $H_1$ is a member of $\mathcal{C}$ then the image of $\phi$ is also a member of $\mathcal{C}$).  Then any group $G$ with $G \cong \pi_1(\mathcal{G})$ for some diagram of groups $\mathcal{G}$ with vertex groups belonging to $\mathcal{C}$, there is a graph of groups (with injective edge homomorphisms) $\mathcal{G'}$ with $G \cong \pi_1(\mathcal{G'})$ with the vertex groups of $\mathcal{G}$ contained in $\mathcal{C}$.  
\end{lemma}

\begin{proof}
	Let $v_0$ be a vertex for the underlying graph of $\mathcal{G}$ and let $T$ be a choice of a maximal tree for the underlying graph of $\mathcal{G}$.  We will always assume that $\pi_1(\mathcal{G})$ is using the base-point $v_0$.  Then using $T$ we obtain maps $G_v \to \pi_1(\mathcal{G})$ and $G_e \to \pi_1(\mathcal{G})$.  We define $\mathcal{G}'$ to have the same underlying graph as $\mathcal{G}$ but with vertex groups $G_v' = \im(G_v \to \pi_1(\mathcal{G}))$, edge groups $G_e' = \im(G_e \to G_{\partial_0 e} \to \pi_1(\mathcal{G})$, and with edge homomorphisms the inclusions $G_v' \to G_{\partial_0 e}'$.  It remains to show that $\pi_1(\mathcal{G}') \cong \pi_1(\mathcal{G})$.  

	Let $\mathcal{X}$ be a graph of spaces for $\mathcal{G}$.  Then we can obtain a graph of spaces $\mathcal{X}'$ for $\mathcal{G}$ as follows.  For each vertex $v$, we take $X_v'$ to be the result of attaching 2-disks to each element of $\pi_1(X_v)$ that is trivial in $\pi_1(\mathcal{G})$, and similarly, we define $X_e'$.  We know that if a curve $\gamma$ in $X_v$ is null-homotopic in the total space of $\mathcal{X}$, then $f_e(\gamma)$ is also null-homotopic in  the total space of $\mathcal{X}$ (since $\gamma$ and $f_e(\gamma)$ are homotopic in the total space of $\mathcal{X}$).  So for the edge maps $f_e'$, we can just take $f_e'|_{X_v} = f_e$ and extend over the added disks to all of $X_e$ arbitrarily, which we can do by the preceding sentence.  Then the total space of $\mathcal{X}'$ has fundamental group $\pi_1(\mathcal{G}')$ and is the result of attaching 2-disks to the total space of $\mathcal{X}$ along null-homotopic curves.  Therefore, $\pi_1(\mathcal{G}') \cong \pi_1(\mathcal{G})$, as desired.   
\end{proof}

The following lemma is contained in the discussion following Lemma 7.4 in \cite{scott_wall}; we include a formal statement and proof for completeness.

\begin{lemma}\label{lem:contraction}
	Suppose that $\mathcal{G}$ is a graph of groups and $e$ is a non-loop edge in the underlying graph of $\mathcal{G}$ where the edge homomorphism $f_e$ is an isomorphism.  Let $\Gamma$ be the underlying graph of $\mathcal{G}$ and let $\Gamma/e$ be the graph obtained by contracting $e$.  Make $\Gamma/e$ into a graph of groups as follows.  Every vertex in $\Gamma/e$ that was not incident to $e$ in $\Gamma$ is given the same vertex graph as in $\mathcal{G}$ while the vertex in $\Gamma/e$ coming from contracting $e$ is labeled by $G_{\partial_1 e}$.  The edges in $\Gamma/e$ with both edges not incident to the vertex coming from $e$ or only incident to $\partial_1 e$ in $\Gamma$ are given the same edge group and edge maps as in $\mathcal{G}$.  The edges $e'$ in $\Gamma/e$ that are incident to $\partial_0 e$ in $\Gamma$ are given the same edge group and are labeled by the same edge maps if the initial vertex of $e'$ is not the vertex corresponding to $e$ and by $f_{\overline{e}}\circ f_e^{-1} \circ f_{e'}$ if the initial vertex of $e'$ is the vertex corresponding to $e$.  Denote this new graph of groups by $\mathcal{G}/e$.  Then $\pi_1(\mathcal{G}) \cong \pi_1(\mathcal{G}/e)$.  
\end{lemma}

\begin{proof}
To get $\pi_1(\mathcal{G})$, we build an appropriate graph of spaces $\mathcal{X}$ and take the fundamental group of the total space of $\mathcal{X}$.  Let $v$ and $u$ be the two vertices incident to $e$ with $f_e : G_e \to G_u$ an isomorphism.  Then we can build $\mathcal{X}$ so that the vertex spaces $X_u$ and the edge space $X_e$ are equal and the edge map $f_e : X_e \to X_u$ that realizes the isomorphism $f_e$ is the identity.  Then by compressing $X_e$ into $X_v$, which is a homotopy equivalence and therefore does not change the fundamental group, we obtain a space with fundamental group equal to $\pi_1(\mathcal{G}/e)$.  
\end{proof}

The assumption that the underlying graph for all of our graphs of groups and diagrams of groups is finite is used in the arguments below when we apply Lemma \ref{lem:contraction} to a spanning tree in a graph of groups.  This lemma cannot be applied to an infinite spanning tree.  One example to keep in mind that illustrates this is the group of dyadic rationals which can be described as the fundamental group of a graph of groups where the underlying group is has vertices $v_i$ indexed by $1,2,3,...$ and edges $e_i$ joining $v_i$ and $v_{i+1}$.  Each vertex $v_i$ and edge $e_j$ is labeled with $\mathbb{Z}$, the edge maps $f_e : G_{e_i} \to G_{v_i}$ are isomorphisms, and the edge maps $f_{\overline{e}}: G_{e_i} \to G_{v_{i+1}}$ are given by multiplication by 2.  This description is also given in \cite{scott_wall}.

A group $G$ is said to have property FA if for any action of $G$ on any tree $X$, there exists a vertex of $X$ fixed by all of $G$ (see \cite{trees}).  A group $G$ is called splittable if $G \cong A *_C B$ where the maps $C \to A$ and $C \to B$ are both not isomorphisms or if $G \cong A *_C$.  Otherwise, $G$ is called unsplittable.  Groups with property FA are unsplittable (see \cite{scott_wall}).   

\begin{lemma} \label{lem:FA}
If $G$ is unsplittable and $\mathcal{G}$ is a graph of groups with $G \cong \pi_1(\mathcal{G})$, then one of the vertex groups of $\mathcal{G}$ is isomorphic to $G$.  
\end{lemma}

\begin{proof}
	We proceed by induction on the number of vertices of the underlying graph of $\mathcal{G}$.  For a one-vertex graph, if there are no edges, the result is immediate.  If there are edges, then by choosing one of the edges $e$, we have $G \cong \pi_1(\mathcal{G}-e) *_{G_e}$.  So assume that the underlying graph of $\mathcal{G}$ has more than one vertex and let $v$ be a vertex.  If there is more than one edge from $v$ to the other vertices of $\mathcal{G}$, than by choosing one of the edges $e$, we again have $G \cong \pi_1(\mathcal{G}-e) *_{G_e}$.  Therefore, we can assume that there is only one edge $e$ from $v$ to the other vertices of $\mathcal{G}$.  Letting $\mathcal{G}_v$ denote the graph of groups of the subgraph containing just the vertex $v$ and all of the loops at $v$ in $\mathcal{G}$, we have $G \cong \pi_1(\mathcal{G}_v) *_{G_e} \pi_1(\mathcal{G}-v)$ and therefore, one of the edge homomorphisms of $e$ is an isomorphism.  But then using Lemma \ref{lem:contraction} allows us to contract $e$ and obtain a new graph of groups $\mathcal{G}/e)$ with one fewer vertex than $\mathcal{G}$ and the set of vertex groups of $\mathcal{G}/e)$ are a subset of the vertex groups of $\pi_1(\mathcal{G})$.  The result then follows by induction.  
\end{proof}

The rank of a group is the minimal number of elements needed to generate the group.  

\begin{theorem}\label{thm:glue}
Let $\mathcal{C}$ be a finite set of groups.  Then there are infinitely many finitely presented groups that are not isomorphic to $\pi_1(\mathcal{G})$ for some diagram of groups $\mathcal{G}$ whose vertex groups all belong to $\mathcal{C}$.  
\end{theorem}

\begin{proof}
	Let $G$ be a finitely presented unsplittable group - for example $\SL_3(\mathbb{Z})$ which has property FA and is therefore unsplittable (see \cite{trees}).  Note that if a finite index subgroup of a group has property FA, then the whole group has property FA (see 6.3.4 in \cite{trees}).  Therefore, the group $G_m = G \times H_m$ where $H_m$ is a finite group with rank greater than or equal to $m$ (for example, the direct product of $n$ copies of $\mathbb{Z}/2\mathbb{Z}$) is unsplittable and has rank at least $m$.   

	Now let $\mathcal{C}_n$ denote the class of all groups that have rank less than or equal to $n$.  This class is closed under homomorphic images since if a group $G$ surjects a group $H$, the rank of $G$ is greater than or equal to the rank of $H$.  The set $\mathcal{C}$ is contained in the class $\mathcal{C}_N$ for some $N$ sufficiently large (namely, taking $N$ to be the maximum rank of all of the ranks of the groups in $\mathcal{C}$. 

	By Lemma \ref{lem:FA}, for $m > n$, the groups $G_m$ cannot be realized as $\pi_1(\mathcal{G})$ for $\mathcal{G}$ a graph of groups with all vertex groups contained in $\mathcal{C}_n$.  Note that, by Lemma \ref{lem:n.i.}, the collection of groups that are isomorphic to $\pi_1(\mathcal{G})$ for some diagram of groups $\mathcal{G}$ with vertex groups in $\mathcal{C}_n$ is equal to the collection of groups that are isomorphic to $\pi_1(\mathcal{G})$ for some graph of groups (with injective edge homomorphisms) $\mathcal{G}$ with vertex groups in $\mathcal{C}_n$.  Then the groups $G_m$ for $m > n$ cannot be isomorphic to $\pi_1(\mathcal{G})$ for any diagram of groups with vertices in $\mathcal{C}$ and the result follows.  
\end{proof}

Martelli asked (\cite{mo_question}) if for $n > 2$ there are could be a finite set of smooth compact $n$-manifolds $\{M_1,...,M_k\}$ such that every closed smooth $n$-manifold can be obtained by gluing together some number of copies of the manifolds $M_i$ with arbitrary diffeomorphisms between boundary components (where different boundary components of a fixed copy of one of the $M_i$ can be glued to different $M_j$ and to each other).  This is equivalent to the question of realizing every compact smooth $n$-manifold as the total space of a graph of spaces where there are a finite number of possible smooth $n$-manifolds for the vertex spaces, and the edge maps are all diffeomorphisms.  For $n =1$ this can be done using the set consisting of just an interval, and for $n$ this can be done by taking a set consisting of a pair of pants, a disk, and a M\"obius band.  Agol shows that no such finite set exists for $n=3$ (see the discussion at \cite{mo_question}) and Freedman introduced the concept of group width in \cite{freedman_width} and proved that no such finite set of manifolds can exist if $n \geq 4$.  Since every finitely presentable group is the fundamental group of some compact smooth $n$-manifold for every $n \geq 4$ (since every 2-complex can be embedded in $\mathbb{R}^5$), Theorem \ref{thm:glue} gives an alternative proof of Freedman's result that there is no such finite set of manifolds for $n \geq 4$.  

As mentioned by Freedman in \cite{freedman_width}, the question of Martelli in the case of simply-connected $n$-manifolds is to our knowledge open for $n \geq 4$.  In the case of $n=4$, this might be especially interesting.  Donaldson proved \cite{donaldson83} that every simply-connected closed smooth 4-manifold with definite intersection form has a diagonalizable intersection form.  Freedman proved \cite{freedman82} that two simply-connected closed smooth 4-manifolds are homeomorphic if and only if they have isomorphic intersection forms.  Indefinite symmetric bilinear form on finitely generated free abelian groups are classified up to isomorphism by rank, signature, and type (even or odd) (\cite{milnor_husemoller}).  The signature of a symmetric bilinear form on a finitely generated free abelian group is divisible by 8 (\cite{milnor_husemoller}).

Putting this together it follows that every smooth simply-connected closed 4-manifold has intersection form isomorphic to either $k[1] + l[-1]$ or $kH + lE_8$ where $H$ is the hyperbolic intersection form and $E_8$ is the intersection form of Freedman's $E_8$ manifold (which we also denote by $E_8$) \cite{freedman82}.  Then by Freedman's work, every smooth simply-connected closed 4-manifold is homeomorphic to a manifold obtained by glueing together copies of $B^4$, $\mathbb{C}P^2 - 2B^4$, $\overline{\mathbb{C}P}^2 - 2B^4$, $E_8 - 2B^4$, and $S^2 \times S^2 - 2B^4$ (where $X^4- 2B^4$ denotes the result of removing two disjoint 4-balls from a 4-manifold $X$).  Thus in the ``mixed case'' where we are allowed topological manifolds $M_i$, we can realise all smooth closed simply-connected 4-manifolds.  Martelli asked specifically about this simply-connected case with smooth manifolds $M_i$ is his original question (see \cite{mo_question}) and Freedman remarks that this case ``looks difficult'' (see \cite{freedman_width}).

\section{Geometric ranks of fundamental groups of graphs of groups}\label{sec:rank}

The \emph{geometric rank} of a group $G$ is the largest $n$ such that $\mathbb{Z}^n$ embeds in $G$.  In this section, we prove a result bounding the geometric rank of the fundamental group of a finite graph of groups in terms of the geometric ranks of the vertex groups.

The following lemma follows from the results in section 3 of \cite{scott_wall} (see the remark following Theorem 3.7 and the discussion following Subgroup Theorem 3.14):

\begin{lemma} \label{lem:subgroups}
Suppose that $\mathcal{G}$ is a graph of groups and $H$ is a subgroup of $\pi_1(\mathcal{G})$.  Then $H \cong \pi_1(\mathcal{H})$ for some graph of groups $\mathcal{H}$ where the vertex groups of $\mathcal{H}$ are subgroups of the vertex groups of $\mathcal{G}$, and where the edge groups of $\mathcal{H}$ are subgroups of the edge groups of $\mathcal{G}$.  
\end{lemma}

\begin{theorem} \label{thm:g.rank}
	Let $\mathcal{G}$ be a graph of groups such that all of the vertex groups of $\mathcal{G}$ have geometric rank less than or equal to $n$.  Then the geometric rank of $\pi_1(\mathcal{G})$ is less than or equal to $n+1$.  
\end{theorem}

\begin{proof}
	Let $A$ be a free abelian subgroup of $\pi_1(\mathcal{G})$ with maximal rank.  Then, by Lemma \ref{lem:subgroups}, $A = \pi_1(\mathcal{A})$ for some graph of groups $\mathcal{A}$ where all of the vertex groups in $\mathcal{A}$ are subgroups of the vertex groups of $\mathcal{G}$.  Since all of the vertex groups of a graph of groups are embedded in the fundamental group of the graph of groups, every vertex group of $\mathcal{A}$ is a free abelian group.  It then suffices to prove the theorem where $\mathcal{G}$ is a graph of groups with every vertex group being a finitely generated free abelian group.  We assume this from now on.  

	Let $e$ be a non-loop edge in the underlying graph of $\mathcal{G}$ with incident vertices $u$ and $v$.  We now show that one of the edge homomorphisms $f_e$ or $f_{\overline{e}}$ must be an isomorphism.  Assume for sake of contradiction that both $f_e : G_e \to G_u$ and $f_{\overline{e}}: G_{\overline{e}} \to G_v$ are not isomorphisms, and therefore, since they are both injective, $f_e$ and $f_{\overline{e}}$ are both not surjective.  Let $g_u \in G_u$ be an element not in the image of $f_e$ and $g_v \in G_v$ be an element not in the image of $f_{\overline{e}}$.  Take the base-point for $\mathcal{G}$ to be $u$.  Then we have the elements $g_u \in \pi_1(\mathcal{G}, u)$ and $e g_v e^{-1} \in \pi_1(\mathcal{G},u)$ and by Theorem 11 in \cite{trees} we have 
$$
	[g_u, e g_v e^{-1}] \neq 1	
$$
	where $[g,h] = ghg^{-1}h^{-1}$ is the commutator of $g$ and $h$.  But $\pi_1(\mathcal{G})$ was assumed to be abelian, and thus we have the desired contradiction.  Thus we may assume (by replacing $e$ with $\overline{e}$ if necessary), that $f_e$ is an isomorphism.  

	Now by choosing a maximal tree $T$ in the underlying graph of $\mathcal{G}$, we may contract $T$ to a single vertex using Lemma \ref{lem:contraction} since by the argument in the previous paragraph, each edge of $T$ has at least one edge homomorphism that is an isomorphism.  Thus we now assume that the underlying graph of $\mathcal{G}$ is a one vertex graph.  

	We now argue that $\mathcal{G}$ can have only one edge.  For if the underlying graph of $\mathcal{G}$ had two distinct edges $e_1$ and $e_2$ then we have the corresponding elements $e_1, e_2 \in \pi_1(\mathcal{G})$ and again by Theorem 11 in \cite{trees}, 
$$
	[e_1,e_2] \neq 1
$$
	thus contradicting the assumption that $\pi_1(\mathcal{G})$ is abelian. Thus we now assume that the underlying graph of $\mathcal{G}$ has exactly one vertex and exactly one edge (if there were no edge, then the result follows immediately).   Therefore, we are the case where $\pi_1(\mathcal{G})$ is an HNN-extension of the vertex group of $\mathcal{G}$.  

	We now argue that both of the edge morphisms of $\mathcal{G}$ must be isomorphisms.  Let $\alpha_1,\alpha_2 : G_e \to G_v$ be the two edge homomorphsims of $\mathcal{G}$.  Suppose for sake of contradiction that there exists an element $a \in G_v$ that is not in the image of $\alpha_1$.  Then, again using Theorem 11 of \cite{trees}, we have 
$$
aea^{-1}e^{-1} \neq 1
$$
in $\pi_1(\mathcal{G})$ which contradicts $\pi_1(\mathcal{G})$ being abelian.  Therefore, $\alpha_1$ is an isomorphism and similarly, $\alpha_2$ is an isomorphism.  

We now argue that both the edge isomorphisms $\alpha_1$ and $\alpha_2$ are in fact the same isomorphism.  Note that in $\pi_1(\mathcal{G})$ we have the relation
$$
e \alpha_1(c) = \alpha_2(c) e
$$
for all $c \in G_e$.  For sake of contradiction, assume that $\alpha_1(c) \neq \alpha_2(c)$ for some $c \in G_e$.  We have 
\begin{align*}
	\alpha_1(c) e \alpha_1(c)^{-1} e^{-1} &= \alpha_1(c) e e^{-1} \alpha_2(c)^{-1} \\
					      &= \alpha_1(c) \alpha_2(c)^{-1} \\
					      &\neq 1
\end{align*}
in $\pi_1(\mathcal{G})$, again by Theorem 11 in \cite{trees}.   

	So now we have the case where the underlying graph of $\mathcal{G}$ is a single vertex together with a single edge, and both of the edge maps are the same isomorphism.  Then, if $v$ is the vertex of $\mathcal{G}$ and $X_v$ is a connected space with $\pi_1(X) \cong G_v$, we can construct a graph of spaces $\mathcal{X}$ for $\mathcal{G}$ by also taking $X_e = X_v$ and taking both of the edge homomorphisms to be the identity.  Therefore, the resulting total space is $X_v \times S^1$ and so we have
	$$
	\pi_1(\mathcal{G}) \cong G_v \times \mathbb{Z}
	$$
and the result follows.  

\end{proof}

Theorem \ref{thm:g.rank} does not hold for diagrams of groups.  To see this, consider for example the genus 3 Heegaard splitting of $T^3$.  Considering this as the total space of a graph of spaces, we have two handlebodies glued along a surface of genus 3.  Note that the fundamental group of a handlebody is free and thus the two vertex groups have geometric rank $1$, however, $\pi_1(T^3) \cong \mathbb{Z}^3$.

\bibliography{graph_of_groups}
\bibliographystyle{alpha}

\end{document}